\documentclass[reqno]{amsart}
\usepackage{amsfonts,amsmath}
\usepackage{amssymb}
\usepackage[cp1251]{inputenc}
\usepackage{graphicx}
\newtheorem{theorem}{Theorem}
\newtheorem*{theorem*}{Theorem}

\theoremstyle{definition}

\newtheorem{rem}{Remark}

\UseRawInputEncoding
\usepackage[english]{babel}

\begin{document}

\title[An exact traveling waves solution] {An exact traveling waves solution  for a special class of semilinear equations}

\author[Dekhnich]{K.A. Dekhnich$^1$}

\address[1]{Mathematics and Mechanics Department, Lomonosov Moscow State University, Leninskie Gory, Moscow, 119991, Russian Federation}
\email{kirill.dekhnich@math.msu.ru}
\begin{abstract} We construct exact traveling-wave solutions for a semilinear parabolic equation with a specific nonlinearity. This nonlinearity arises, in particular, in voting models based on branching Brownian motion. Previously known results are special cases of those presented here.
\end{abstract}

\maketitle

\section{Introduction}

This paper is devoted to finding exact traveling-wave solutions to semilinear parabolic equations with a special type of nonlinearity associated with a probabilistic interpretation. Such equations were considered in a generalized form in the work of An, Henderson, and Ryzhik \cite{AnHendersonRyzhik2023} from the point of view of voting models on a branching Brownian motion tree \cite{Skorohod}, as a result of which a connection was established between branching Brownian motion and the semilinear equation
\begin{equation}\label{eq:0}
    \begin{split}
        &u_t=u_{xx}+f(u),\\
        &u(0,x)=g(x).
    \end{split}
\end{equation}

The nonlinearity that arises in the random outcome voting model, as shown in \cite{AnHendersonRyzhik2023}, is of the form
\[
    f(u)=\beta\sum_{n=2}^{N}p_n\left(\sum_{k=0}^{n}C_n^k\alpha_{k,n}u^k\left(1-u\right)^{n-k}-u\right),
\]
where $\beta>0$ is a parameter of the branching process, see \cite{AnHendersonRyzhik2023}, $p_k=P(N=k)\ge 0, \quad k\ge2, \quad \sum_{n=2}^{N}p_k=1$ is a set of probabilities corresponding to this branching process, and $0\le\alpha_{k,n}\le1,\text{ } 0\le k\le n, \text{ } n\ge2$ is a fixed set of probabilities, where
\[
\alpha_{0,n}=0, \quad \alpha_{n,n}=1, \quad n\ge2.
\]
\begin{equation*}
    p_k=P(N=k)\ge 0, k\ge2.
\end{equation*} 
The voting rules are set as follows: if a parent has $n$ children and $k$ of them voted for option 1, then the parent will vote for option 1 with probability $\alpha_{k,n}.$

One of the main results of \cite{AnHendersonRyzhik2023} is the following theorem  (p. 11).
\begin{theorem*}
    Let $f(u)$ be a polynomial of degree $N$ and $f(0)=f(1)=0.$ Then there exists a representation of the random outcome voting model in terms of $N$-ary branching Brownian motion corresponding to equation (\ref{eq:0}) with initial condition $g(x),$ where the function $g(x)$ is continuous and satisfies the condition $0\le g(x)\le 1$ for all $x\in \mathbb{R}^d.$ 
\end{theorem*}
\noindent The nonlinearities considered below satisfy the conditions of this theorem.

For the first time, an equation with such conditions for nonlinearity was considered in the seminal work of Kolmogorov, Petrovskii and Piskunov \cite{Kolmogorov1937}, where the problem of the existence and stability of traveling waves was studied.
The existence of traveling waves for semilinear parabolic equations was also studied by Kanel \cite{Kanel1960}, \cite{Kanel1962}, \cite{Kanel1964} in the spirit of the original paper \cite{Kolmogorov1937}.

An exact analytical solution in the form of a traveling wave was obtained by Ablowitz and Zeppetella \cite{Ablowitz1979} for the Fisher equation \cite{Fisher}
$$u_t=u_{xx}+u\left(1-u\right)$$
with $c=5/\sqrt{6} \text{ and } u(-\infty)=1, \text{ } u(+\infty)=0,$ where $u=u(x-ct).$
In this case, the solution has the form
\begin{equation*}
    u(x-5/\sqrt{6}t)=\frac{1}{\left(1+re^{\frac{x-5/\sqrt{6}t}{\sqrt{6}}}\right)^2}, \quad r>0.
\end{equation*}

Kaliappan \cite{Kaliappan} considered a more general equation
$$u_t=u_{xx}+u-u^n, \,n> 1,$$
and obtained the solution
\begin{equation*}
    u(x-ct)=\frac{1}{\left(1+re^{k(x-ct)}\right)^{\frac{2}{n-1}}}, \quad r>0,
\end{equation*}
where
\begin{equation*}
    k=\frac{n-1}{\sqrt{2n+2}}, \quad c=\frac{n+3}{\sqrt{2n+2}}.
\end{equation*}

Interest in finding exact traveling-wave solutions remains due to their importance in modeling various biological and chemical phenomena. The book \cite{Petrovskii} is commonly cited as a review of the current state of the art; however, it turns out that new results can also be obtained in this traditional field.

It is worth noting a recent paper \cite{Kogan}, which presents an integration method that, among other things, allows one to obtain exact traveling-wave solutions to the KPP equation with a nonlinearity already studied by Ablowitz and Zeppetella \cite{Ablowitz1979}. In particular, this paper presents results from \cite{AnHendersonRyzhik2023} confirming the relevance of the class of nonlinearities considered in this paper.

The search for exact solutions to equations of this class can be carried out within the framework of a different approach, known as the method of simplest equation \cite{Kudryashow}. The essence of this method is that instead of choosing a specific ansatz, the solution is represented as a polynomial in a function satisfying some auxiliary ordinary differential equation (e.g., Riccati). This approach allows one to find solutions of a more complex structure, not necessarily monotone. In a recent paper \cite{McCue}, a method based on the use of Weierstrass elliptic functions was applied to the Fisher equation \cite{Fisher}.

\section{Main result}
\begin{theorem}
    There is an exact solution to the equation
    \begin{equation}\label{eq_main}
       u_t(t,x)=u_{xx}(t,x)-\sum_{j=2}^{n}p_ju^j(t,x)+u(t,x) 
    \end{equation}
    in the form
    $$u(x\mp ct)=\frac{1}{\left(1+p\exp\left({\pm k\left(x\mp ct\right)}\right)\right)^{\frac{2}{n-1}}}, \quad p>0,$$
    with
    \[
     c=\frac{n+2p_n+1}{\sqrt{p_n\left(2n+2\right)}}, \qquad k=\frac{\sqrt{p_n}\left(n-1\right)}{\sqrt{2n+2}},
    \]
    where $p_n=1$ and $p_j=0,$ $j=2,...,n-1$ in the case of even $n$ and $p_n>0,$ $p_n+p_{\frac{n+1}{2}}=1,$ $p_j=0, j=2,...,n-1; j\ne\frac{n+1}{2}$ in the case of odd $n.$
\end{theorem}

{\bf Proof.}
Equation \eqref{eq_main} for
 $u=u(z)$, $\,z= x-ct$, takes the form
\begin{equation}\label{eq:54}
    u''+cu'-\sum_{j=2}^{n}p_ju^j+u=0.
\end{equation}
As in [7], we seek a solution in the form of the following ansatz:
\begin{equation}\label{eq:55}
    u(z)=\frac{1}{\left(1+pe^{kz}\right)^\alpha}, \quad p>0,
\end{equation}
where $\alpha=\frac{2}{n-1}$ due to the balance of degrees. Then
\begin{equation*}\label{eq:56}
    u'(z)=-\frac{\alpha kpe^{kz}}{\left(1+pe^{kz}\right)^{\alpha+1}},\quad
    u''(z)=\frac{\alpha k^2pe^{kz}\left(\alpha pe^{kz}-1\right)}{\left(1+pe^{kz}\right)^{\alpha+2}},\quad
    u^n(z)=\frac{1}{\left(1+pe^{kz}\right)^{\alpha+2}}.
\end{equation*}
After substituting the latter computations and (\ref{eq:55})
into (\ref{eq:54}) we obtain 
\begin{eqnarray}\label{eq:60}
 &&\alpha \,k^2\, p\, e^{kz}\,(\alpha p \,e^{kz}-1)-c \alpha k p\, e^{kz} (1+p\, e^{kz})-p_n  \nonumber\\&& -\sum_{j=2}^{n-1}p_j\left(1+p\, e^{kz}\right)^\frac{2(n-j)}{n-1}+\left(1+e^{kz}\right)^2=0.
\end{eqnarray}
To eliminate fractional powers, we denote
$
l=\left(1+pe^{kz}\right)^\frac{1}{n-1},
$
from which it follows that $$pe^{kz}=l^{n-1}-1.$$ Substituting this equality into (\ref{eq:60}), we get
\begin{eqnarray*}
&&\alpha k^2\left(\alpha l^{2(n-1)}-\left(2\alpha+1\right)l^{n-1}+\left(\alpha+1\right)\right)-c \alpha kl^{2(n-1)}+c\alpha kl^{n-1}\\
&&-p_n-\sum_{j=2}^{n-1}p_jl^{2(n-j)}+l^{2(n-1)}=0.
\end{eqnarray*}
Let us equate the coefficients of the powers $l$ to zero. For $l^0$ we get $\alpha\left(\alpha+1\right)k^2-p_n=0$, that is
\begin{equation}\label{eq:61}
    k^2=\frac{p_n}{\alpha\left(\alpha+1\right)},
\end{equation}
from which it follows that $p_n>0.$ Considering that $\alpha=\frac{2}{n-1}$, we obtain
\begin{equation}\label{eq:62}
    k_1=\frac{\left(n-1\right)\sqrt{p_n}}{\sqrt{2n+2}}, \quad k_2=-\frac{\left(n-1\right)\sqrt{p_n}}{\sqrt{2n+2}}.
\end{equation}
For the degree $l^{2(n-1)}$ not contained in the sum
\begin{equation}\label{eq:sum}
     \sum_{j=2}^{n-1}p_jl^{2(n-j)},
\end{equation}
we get $\alpha^2k^2-c\alpha k+1=0$, that is
\begin{equation}\label{eq:64}
     c=\frac{\alpha^2k^2+1}{\alpha k}.
\end{equation}
Now we substitute (\ref{eq:62}) here and get two cases
\[
c_1=\frac{n+2p_n+1}{\sqrt{p_n\left(2n+2\right)}}, \quad c_2=-\frac{n+2p_n+1}{\sqrt{p_n\left(2n+2\right)}}.
\]
For power $l^{n-1}$, the factor can contain a term contained in the sum \eqref{eq:sum} if $n-1=2(n-j),$, that is, $j=\frac{n+1}{2}.$ The number $j$ must be an integer, which is possible only for odd $n$. Therefore, two cases must be considered.
\\
\par
\textbf{a) n is even.} Then, equating the coefficient of $l^{n-1}$ to zero, we have
$$
-\alpha k^2\left(2\alpha+1\right)+c\alpha k=0.
$$
Taking into account (\ref{eq:61}) and (\ref{eq:64}), we obtain
$
\frac{p_n}{\alpha(\alpha+1)}\alpha(\alpha+1)=1,
$
whence
$ p_n=1.$
Since other degrees occur in the sum only once, then $p_j=0,$ $j=2,...,n-1.$
\\
\par
\textbf{b) n is odd;} In this case,
\[
-\alpha k^2\left(2\alpha+1\right)+c\alpha k-p_{\frac{n+1}{2}}=0.
\]
Taking into account (\ref{eq:61}) and (\ref{eq:64}), we obtain
$
\frac{p_n}{\alpha(\alpha+1)}\alpha(\alpha+1)=1-p_{\frac{n+1}{2}}, 
$
that is $p_n+p_{\frac{n+1}{2}}=1.$
Since other powers occur in the sum \eqref{eq:sum} only once, then $p_j=0, j=2,...,n-1; j\ne\frac{n+1}{2}.$
\par
Choosing $(k_1,c_1)$ or $(k_2,c_2)$ results in traveling waves satisfying $u(-\infty)=1,$ $u(+\infty)=0$ or $u(-\infty)=0,$ $u(+\infty)=1$, respectively.
The theorem is proven.

\medskip
\begin{rem} It should be noted that in the case of even $n$ the formulas obtained in Theorem 1 coincide with the results of \cite{Ablowitz1979} and \cite{Kaliappan}, whereas for odd $n$ new possibilities for exact solutions appear.
\end{rem}

\medskip

Note that the described method can also be applied to other nonlinearities, both local (for example, Allan-Cahn nonlinearities \cite{Cahn}) and nonlocal (for example, Gourley \cite{Gourley}).

\section*{Acknowledgments}
Supported by ``Vega Institute Foundation''.

\end{document}